\documentclass[12pt,english]{extarticle}
\usepackage[utf8]{inputenc}
\usepackage[a4paper]{geometry}
\usepackage{fancyhdr}
\usepackage{color}
\usepackage{verbatim}
\usepackage{amsmath}
\usepackage{amsthm}
\usepackage{amssymb}

\makeatletter
\numberwithin{equation}{section}
\numberwithin{figure}{section}
\theoremstyle{plain}
\newtheorem{thm}{\protect\theoremname}[section]
\theoremstyle{plain}
\newtheorem{lem}[thm]{\protect\lemmaname}
\theoremstyle{definition}
\newtheorem{defn}[thm]{\protect\definitionname}
\theoremstyle{plain}
\newtheorem{prop}[thm]{\protect\propositionname}
\theoremstyle{remark}
\newtheorem{rem}[thm]{\protect\remarkname}

\usepackage[english]{babel}
\usepackage{amsthm}
\usepackage{amsfonts}

\usepackage{hyperref}

\usepackage{color}

\definecolor{rot}{rgb}{1.000,0.000,0.000}

\usepackage{tikz}

\usepackage{tkz-graph}
\GraphInit[vstyle = Shade]
\tikzset{
  LabelStyle/.style = {rectangle, rounded corners, draw,
                        minimum width = 2em, 
                        text = black, font = \bfseries },
  VertexStyle/.style = {circle, draw, 
                                font =  \large\bfseries},
  EdgeStyle/.style = {->, bend left} }
\newcommand{\N}{\mathbb{N}}
\newcommand{\R}{\mathbb{R}}

\newcommand{\cE}{\mathcal{E}}

\newcommand{\cD}{\mathcal{D}}

\newcommand{\e}{\mathrm{e}}

\renewcommand{\L}{\mathrm{L}}

\renewcommand{\d}{\mathrm{d}}

\definecolor{rot}{rgb}{1.000,0.000,0.000}

\newcommand{\rra}{\rangle\!\rangle}
\newcommand{\lla}{\langle\!\langle}
\newcommand{\one}{1\!\!1}

\makeatother

\usepackage{babel}
\providecommand{\definitionname}{Definition}
\providecommand{\lemmaname}{Lemma}
\providecommand{\propositionname}{Proposition}
\providecommand{\remarkname}{Remark}
\providecommand{\theoremname}{Theorem}

\begin{document}
\title{Coarse-graining and reconstruction for Markov matrices\thanks{Research supported by DFG via SFB 1114 (project no.235221301, subproject C05).}}
\author{Artur Stephan\thanks{Weierstraß-Institut für Angewandte Analysis und Stochastik, Mohrenstraße 39, 10117 Berlin, Germany, e-mail: \tt{artur.stephan@wias-berlin.de}}}
\date{\today}
\maketitle

\lhead{Coarse-graining and reconstruction}

\rhead{Artur Stephan}

\chead{\today}
\begin{abstract}
We present a coarse-graining (or model order reduction) procedure
for stochastic matrices by clustering. The method is consistent with
the natural structure of Markov theory, preserving positivity and
mass, and does not rely on any tools from Hilbert space theory. The
reconstruction is provided by a generalized Penrose-Moore inverse
of the coarse-graining operator incorporating the inhomogeneous invariant
measure of the Markov matrix. As we show, the method provides coarse-graining
and reconstruction also on the level of tensor spaces, which is consistent
with the notion of an incidence matrix and quotient graphs, and, moreover,
allows to coarse-grain and reconstruct fluxes. Furthermore, we investigate
the connection with functional inequalities and Poincaré-type constants.

\vspace{0.3cm}

\end{abstract}

\section{Introduction}

Coarse-graining or model reduction is a fundamental procedure that
reduces the complexity of a physical model. It is a well-established
tool used in many branches of applied mathematics including analysis,
modeling and numerics. In this paper we are interested in coarse-graining
for physical systems on a finite state space described by Markov matrices.

Let us first describe the mathematical setting. Fixing a finite state
spaces $\mathcal{Z}=\left\{ 1,\dots,n\right\} $, $n\in\N$, the statistical
states are given by the set of probability vectors
\[
\mathrm{Prob}(\mathcal{Z})=\left\{ p\in\R^{n}:p_{i}\geq0,\sum_{i=1}^{n}p_{i}=1\right\} \subset\R^{n}:=X^{*}.
\]
Important is the distinction between primal $X$ and dual spaces $X^{*}$
(although both are isomorphic to the $\R^{n}$ as real-vector spaces),
where the first contains functions on $\mathcal{Z}$ equipped with
the supremum-norm, and the second contains probabilities on $\mathcal{Z}$
equipped with the 1-norm. Dual pairing is denoted by $\langle\cdot,\cdot\rangle$.
Apart from $X$ being a vector space, it has a natural order, i.e.
$x\geq y$ if the pointwise estimate $x_{i}\geq y_{i}$ holds for
all $i\in\mathcal{Z}$. Moreover, it is an algebra, i.e. product of
two elements $x$ and $y$ is given by $\left(x\cdot y\right)_{i}=x_{i}\cdot y_{i}$
by pointwise multiplication.

The change of statistical states is described by a Markov matrix (or
operator) $K:X\to X$, which, by definition, satisfies $K_{ij}\geq0$
and $\sum_{j\in\mathcal{Z}}K_{ij}=1$ for all $i\in\mathcal{Z}$.
Equivalently, $K$ maps non-negative elements in $X$ to non-negative
and $K1\!\!1=1\!\!1$, where $\one_{X}=(1,\dots,1)^{\mathrm{T}}$
is the constant one-vector. The invariant vector $\pi\in\mathrm{Prob}(\mathcal{Z})\subset X^{*}$
of $K^{*}$ is defined by satisfying $\sum_{i}\pi_{i}K_{ij}=\pi_{j}$,
or, equivalently, by $\pi^{T}K=\pi^{T}$, or $K^{*}\pi=\pi$. Throughout
the paper, we assume that the invariant vector is unique and positive.
(See e.g. \cite{Norr97MC} for introductory reading on Markov matrices.)
We remark that for further generalizations to infinite and continuous
state spaces (in which case $X$ and $X^{*}$ are infinite dimensional)
we denote by $B^{*}:Y^{*}\to X^{*}$ the dual (or adjoint) operator
of $B:X\to Y$, which is just the (real) transpose in matrix representation.

The aim of the paper is to present an operator-theoretic and structure
preserving approach to coarse-graining for Markov matrices. In contrast
to classical model order reduction procedures, it does not rely on
tools from Hilbert space theory like orthonormal projections or symmetry,
as for example Krylov subspace projection methods and so on (see e.g.
\cite{SVR08MOR}). In general, Hilbert space projections will not
preserve positivity of the measures, which is unphysical. Here, the
approach is more direct and based on clusters. Reduction methods based
on clusters as \cite{ChSc20CBMR,CYRS20ROMNS} preserve the graph theoretical
structure for Markov chains, but do not distinguish between primal
and dual spaces. With that, the theory is based on homogenous Euclidean
spaces, which are not canonical for Markov matrices as the invariant
measure $\pi\in\mathrm{Prob}(\mathcal{Z})$ is in general not homogenous
and the natural Hilbert space would be $\L^{2}(\pi)$.

Here, the coarse-graining procedure is based on the structural duality
between $X$ and $X^{*}$. Implicitly, the involved operators (maybe
in a modified form) have been used in literature (see e.g. \cite{ChSc20CBMR,PMK06RedCME})
and we recall technical results from \cite{MieSte19ECLRS} in Section
\ref{sec:Coarse-graining}. However, to the author's knowledge no
structural study has been done so far. Heuristically, the reconstruction
procedure rebuilds the information from the coarser system back to
the finer system using the local information of the invariant probability
vector $\pi$. The reconstruction operators can be understood as generalized
Penrose-Moore inverse of the coarse-graining operator respecting the
(in general inhomogeneous) invariant measure $\pi$. This procedure
has many different mathematical advantages, which become clear in
the following. First, it is consistent with the graph-theoretic notion
of a incidence matrix (or operator) and quotient graph (see Section
\ref{sec:Coarse-grained-network}). Moreover, it provides tools for
reconstructing functions on tensor spaces (e.g. fluxes defined on
edges) as well (see Section \ref{subsec:Flux-reconstruction} for
more details).

Another important question regarding coarse-graining is the question
how eigenvalues depend on the reduction procedure. The first nontrivial
eigenvalue defines the spectral gap and provides information regarding
asymptotical decay of the process (see e.g. \cite{BoTe06LSI}). In
the last decades, functional inequalities for Markov processes on
discrete states spaces have been studied intensively \cite{BoTe06LSI,ErbMaa12RCFM,John2017DLSI,FatShu18CTUMCDS,FaFa21SOSLS}.
In Section \ref{sec:FunctiIn-PoincareConstants}, we derive the connection
between coarse-graining and energy functionals. In particular, we
derive estimates for functional inequalities and discrete Poincaré-type
constants (like the Poincaré constant, or log-Sobolev constant).

\section{Coarse-graining\label{sec:Coarse-graining}}

We present the operator theoretic framework for capturing the collection
of states, which has also been introduced in \cite{MieSte19ECLRS}.

\subsection{Operator theoretic coarse-graining}

For two finite state spaces $\mathcal{Z}=\left\{ 1,\dots,n\right\} $,
$\hat{\mathcal{Z}}=\left\{ 1,\dots,\hat{n}\right\} $ with $\hat{n}<n$,
we assume that there is a given surjective function $\phi:\mathcal{Z}\to\hat{\mathcal{Z}}$,
which plays the role of a coarse-graining or clustering map. We define
the \textit{coarse-graining operator} $M:\hat{X}\to X$ by $(M\hat{x})_{i}=\hat{x}_{\phi(i)}$
for all $\hat{x}\in\hat{X}$. One easily sees that $M$ is a deterministic
Markov matrix since the adjoint (or dual) matrix $M^{*}:X^{*}\to\hat{X}^{*}$
maps pure states (or dirac-measures) to pure states. In fact the dual
operator $M^{*}$ should be called \textit{coarse-graining operator}
because it maps statistical states in $X^{*}$ to coarser states in
$\hat{X}^{*}$. Since $M$ is a deterministic Markov matrix we have
that for all $\hat{x},\hat{y}\in\hat{X}$ it holds $M(\hat{x}\cdot\hat{y})=M\hat{x}\cdot M\hat{y}$,
where the multiplication is meant pointwise. (By the way, this characterizes
all deterministic Markov matrices.)

Fixing a positive probability vector $\pi\in X^{*}$, we may define
the multiplication operator given by the diagonal matrix $Q_{\pi}:X\to X^{*}$,
i.e. $\left(Q_{\pi}x\right)_{i}=\pi_{i}x_{i}$. We observe that $Q_{\pi}$
is symmetric and its inverse is given by $Q_{\pi}^{-1}:X^{*}\to X$,
$p\mapsto\rho=(p_{i}/\pi_{i})_{i}$. One easily sees that the multiplication
operator satisfies $\langle x,Q_{\pi}y\rangle=\langle x\cdot y,\pi\rangle$
for all $x,y\in X$. Although the spaces are finite-dimensional and
isomorphic, we remark that the parameter $\pi$ of $Q_{\pi}$ is an
element of the dual space and the inverse $Q_{\pi}^{-1}:p\mapsto\rho$
maps a probability vector $p$ to the relative density $\rho$ of
$p$ with respect to $\pi$ as a discrete analogue of the Radon-Nikodym
derivative.

We define a new coarse-grained measure $\hat{\pi}$ by $\hat{\pi}=M^{*}\pi$.
We easily observe that $\hat{\pi}$ is also positive. We have the
following characterization of a deterministic Markov matrix.
\begin{lem}[{\cite[Lemma 2.4]{MieSte19ECLRS}}]
\label{lem:DetMarkovMatrix}We have $\hat{\pi}=M^{*}\pi$ if and
only if $Q_{\hat{\pi}}=M^{*}Q_{\pi}M$.
\end{lem}

\begin{proof}
Evaluating $Q_{\hat{\pi}}=M^{*}Q_{\pi}M$ at $\hat{\one}$, we have
$\hat{\pi}=Q_{\hat{\pi}}\hat{\one}=M^{*}Q_{\pi}M\hat{\one}=M^{*}Q_{\pi}\one=M^{*}\pi$,
which is one direction of the claim.

For the other claim, we introduce the multiplication operator $\Pi_{\hat{y}}:\hat{X}\to\hat{X}$
by $\left(\Pi_{\hat{y}}\hat{x}\right)_{j}=\hat{y}_{j}\hat{x}_{j}$
with the dual operator $\Pi_{\hat{y}}^{*}:\hat{X}^{*}\to\hat{X}^{*}$
given by $\Pi_{\hat{y}}^{*}\hat{c}=Q_{\hat{c}}\hat{y}$. Using that
$M$ is a deterministic operator we have $M\Pi_{\hat{y}}=\Pi_{M\hat{y}}M$,
which implies by dualizing $\Pi_{\hat{y}}^{*}M^{*}=M^{*}\Pi_{M\hat{y}}^{*}$.
So we get for any $\hat{x}$ that $Q_{\hat{\pi}}\hat{x}=\Pi_{\hat{x}}^{*}\hat{\pi}=\Pi_{\hat{x}}^{*}M^{*}\pi=M^{*}\Pi_{M\hat{x}}^{*}\pi=M^{*}Q_{\pi}M\hat{x}$.
\end{proof}
This important relation does not hold if $M$ is not a deterministic
operator. Moreover, it implies that the following diagram commutes:

\vspace{-0.5cm}

\begin{figure}[h]
\begin{center} 
\begin{tikzpicture}
\node (n1) at (0,0) {$\hat X^*$};
\node (n2) at (0,1.3) {$\hat X$};
\node (n3) at (2,0) {$X^*$};
\node (n4) at (2,1.3) {$X$};
\draw[ultra thick, ->] (n2)-- node[pos=0.5,above,color=black]{$M$}(n4) ; 
\draw[ultra thick, ->] (n3)-- node[pos=0.5,above,color=black]{$M^*$}(n1) ;
\draw[ultra thick, ->] (n2)-- node[pos=0.5,left,color=black]{$Q_{\hat\pi}$}(n1) ;
\draw[ultra thick, ->] (n4)-- node[pos=0.5,right,color=black]{$Q_{\pi}$}(n3) ;
\end{tikzpicture}\end{center}
\end{figure}

\vspace{-1cm}
\begin{flushleft}
Since $Q_{\hat{\pi}}$ is invertible, it is now possible to ``invert''
the coarse-graining operator $M$, by defining the so-called \textit{reconstruction
operator} $N:X\to\hat{X}$:
\begin{equation}
N=Q_{\hat{\pi}}^{-1}M^{*}Q_{\pi}:X\to\hat{X},\quad N^{*}=Q_{\pi}MQ_{\hat{\pi}}^{-1}:\hat{X}^{*}\to X^{*}.\label{eq:DefinitionN}
\end{equation}
\par\end{flushleft}

\begin{figure}[h]
\begin{center} 
\begin{tikzpicture}
\node (n1) at (0,0) {$\hat X^*$};
\node (n2) at (0,1.3) {$\hat X$};
\node (n3) at (2,0) {$X^*$};
\node (n4) at (2,1.3) {$X$};


\draw[bend left=10,ultra thick,->] (n2) to node[pos=0.5,above,color=black]{$M$}(n4) ; 
\draw[bend left=10,ultra thick,->,blue] (n4) to  node[pos=0.5,below,color=black]{$N$}(n2) ; 
\draw[bend right=10,ultra thick,->] (n3) to node[pos=0.5,above,color=black]{$M^*$}(n1) ;
\draw[bend right=10,ultra thick,->, blue] (n1) to node[pos=0.5,below,color=black]{$N^*$}(n3) ;

\draw[ultra thick,->] (n2)-- node[pos=0.5,left,color=black]{$Q_{\hat\pi}$}(n1) ;
\draw[ultra thick,->] (n4)-- node[pos=0.5,right,color=black]{$Q_{\pi}$}(n3) ;
\end{tikzpicture}\end{center}
\end{figure}

\newpage

The operator $N^{*}$ reconstructs the coarser statistical states
in $\hat{X}^{*}$ respecting the measure $\pi\in X^{*}$. Before summarizing
properties of $N$ in the next proposition, we introduce the notion
of \textit{detailed balance}.
\begin{defn}
\label{def:DetailedBalance}A Markov matrix $K$ is said to satisfy
the $\textit{detailed balance condition}$ with respect to its positive
invariant measure $\pi$, if $K^{*}Q_{\pi}=Q_{\pi}K$.
\end{defn}

\begin{prop}[{\cite[Lemma 2.5, Proposition 2.7]{MieSte19ECLRS}}]
\label{prop:CoarseGrainingProperties}Let $M:\hat{X}\to X$ be a
deterministic Markov matrix and let $\pi\in X^{*}$ be a given positive
probability vector. Let $N$ and $N^{*}$ be defined by \eqref{eq:DefinitionN}.
Then the following holds:
\begin{enumerate}
\item $N$ is a Markov matrix and $N^{*}\hat{\pi}=\pi$, i.e. $N^{*}$ inverts
with respect to $\pi$.
\item $NM=\mathrm{id}_{\hat{X}}$ and $MN=:P$ is a (Markov) projection
on $X$. We have the splitting $X=\mathrm{Range}(P)+\mathrm{Ker}(P)=\mathrm{Range}(M)+\mathrm{Ker}(N)$.
The adjoint $P^{*}$ has $\pi$ as its stationary measure and satisfies
detailed balance.
\end{enumerate}
\end{prop}

\begin{proof}
Clearly $N$ is nonnegative and $N\one=Q_{\hat{\pi}}^{-1}M^{*}Q_{\pi}\one=Q_{\hat{\pi}}^{-1}M^{*}\pi=Q_{\hat{\pi}}^{-1}\hat{\pi}=\hat{\one}$.
Hence, $N$ is a Markov matrix. Moreover we have $N^{*}\hat{\pi}=Q_{\pi}MQ_{\hat{\pi}}^{-1}\hat{\pi}=Q_{\pi}M\hat{\one}=\pi$.

For the second claim, we use Lemma \ref{lem:DetMarkovMatrix}, which
implies that $M^{*}N^{*}=M^{*}Q_{\pi}MQ_{\hat{\pi}}^{-1}=\mathrm{id}_{\hat{X}^{*}}$.
Hence, $NM=\mathrm{id}_{\hat{X}}$ and $P=MN$ is a projection. This
provides the decomposition of $X$ since $N$ is surjective and $M$
is injective. Since it is the composition of Markov matrices $P$
is again a Markov matrix and obviously $P^{*}\pi=\pi$. To see that
$P$ satisfies detailed balance, we observe
\[
Q_{\pi}P=Q_{\pi}MN=Q_{\pi}MQ_{\hat{\pi}}^{-1}M^{*}Q_{\pi}=N^{*}M^{*}Q_{\pi}=P^{*}Q_{\pi}\,.
\]
\end{proof}

\begin{rem}
We remark that in \cite{Step13IMOMDT} inverse operators for general
Markov operators have been introduced and their relation to the direction
of time has been investigated.
\end{rem}

Finally, we investigate the connection between the ``inverse'' operator
$N$ and the Penrose-Moore inverse of linear algebra. First, we see
that $N:X\to\hat{X}$ is a pseudo inverse of $M:\hat{X}\to X$, because
$MNM=M$ and $NMN=N$ by Proposition \ref{prop:CoarseGrainingProperties}.
Recall that for an injective $M:\hat{X}\to X$, the Penrose-Moore
inverse of $M$ can be defined by
\[
M^{+}=(M^{*}M)^{-1}M^{*}.
\]
The next proposition shows, that this formula provides exactly $N$,
if the adjoint operator $M^{*}$ is understood in the space $\L^{2}(\pi)$.
In particular, if $\pi=\frac{1}{N}(1,\cdots,1)^{\mathrm{T}},$ we
have that $N=M^{+}$. To see this, we define the $\L^{2}(\pi)$-inner
product in $X$ by 
\[
(x,y)_{\pi}:=\langle x,Q_{\pi}y\rangle=\langle x\cdot y,\pi\rangle\,.
\]

\begin{prop}
The reconstruction operator $N=Q_{\hat{\pi}}^{-1}M^{*}Q_{\pi}$ is
the $\L^{2}(\pi)$-adjoint of $M$. In particular, $N$ is generalized
Penrose-Moore inverse of $M$ in $\L^{2}(\pi)$.
\end{prop}

\begin{proof}
We have that
\[
(M\hat{x},y)_{\pi}=\langle\hat{x},M^{*}Q_{\pi}y\rangle=\langle\hat{x},Q_{\hat{\pi}}Ny\rangle=(\hat{x},Ny)_{\hat{\pi}}\,.
\]
\end{proof}
We note that the notion of $\textit{detailed balance}$ from Definition
\ref{def:DetailedBalance} means that the Markov matrix $K$ is symmetric
in $\L^{2}(\pi)$.

\subsection{Example\label{subsec:Example}}

For an example, we consider $Z=\left\{ 1,2,3\right\} $ and $\hat{Z}=\{\hat{1},\hat{2}\}$
and define $\phi(1)=\hat{1}$, $\phi(2)=\phi(3)=\hat{2}$. In matrix
representation, the coarse-graining operator has the form
\[
M=\begin{pmatrix}1\\
 & 1\\
 & 1
\end{pmatrix},\quad M^{*}=\begin{pmatrix}1\\
 & 1 & 1
\end{pmatrix}\,.
\]
Setting $\pi=(\pi_{1},\pi_{2},\pi_{3})^{\mathrm{T}}$ we obtain $\hat{\pi}=(\pi_{1},\pi_{2}+\pi_{3})^{\mathrm{T}}$,
and hence
\[
N=\begin{pmatrix}1\\
 & \frac{\pi_{2}}{\pi_{2}+\pi_{3}} & \frac{\pi_{3}}{\pi_{2}+\pi_{3}}
\end{pmatrix}\,,\quad P=MN=\begin{pmatrix}1\\
 & \frac{\pi_{2}}{\pi_{2}+\pi_{3}} & \frac{\pi_{3}}{\pi_{2}+\pi_{3}}\\
 & \frac{\pi_{2}}{\pi_{2}+\pi_{3}} & \frac{\pi_{3}}{\pi_{2}+\pi_{3}}
\end{pmatrix}\,.
\]

\subsection{Coarse-graining for Markov matrices}

Let a Markov matrix $K:X\to X$ be given. We assume that its adjoint
$K^{*}$ has a unique invariant measure $\pi$, i.e. $K^{*}\pi=\pi$.
We define the coarse-grained Markov matrix $\hat{K}$ by contracting
$K$ via

\[
\hat{K}=NKM:\hat{X}\to\hat{X}.
\]
The next proposition shows that $\hat{K}$ can indeed be understood
as a coarse-grained version of $K$.
\begin{thm}
\label{thm:Coarse-grainingMarkovMatrix}Let a Markov matrix $K:X\to X$
with an invariant measure $\pi$ be given. Let $M$ be a deterministic
Markov matrix, and reconstruction operator $N$ be defined by \eqref{eq:DefinitionN}.
Let $\hat{K}:=NKM:\hat{X}\to\hat{X}$. Then, we have
\begin{enumerate}
\item $\hat{K}$ is a Markov matrix on $\hat{X}$.
\item $\hat{K}^{*}$ has $\hat{\pi}$ as ita invariant measure.
\item Define the Markov chain $p_{k+1}=K^{*}p_{k}$, $p_{0}\in X^{*}$.
If there is an equilibration of the form $p_{k}=N^{*}\hat{p}_{k}$
for $\hat{p}_{k}\in\hat{X}$ and all $k\geq0$\textcolor{blue}{, }then
the probability vectors $\hat{p}_{k}$ satisfy the coarse-grained
Markov chain $\hat{p}_{k}=\hat{K}^{*}\hat{p}_{k-1}$.
\item If $K$ satisfies the detailed balance condition with respect to $\pi$,
then does $\hat{K}$ with respect to $\hat{\pi}$.
\end{enumerate}
\end{thm}

Equilibration $p_{k}=N^{*}\hat{p}_{k}$ means that the densities of
$p_{k}$ and $\hat{p}_{k}$ with respect to $\pi$ or $\hat{\pi}$,
respectively are equilibrated, i.e.
\[
p_{k}=Q_{\pi}MQ_{\hat{\pi}}^{-1}\hat{p}_{k}\quad\Leftrightarrow\quad\rho_{k}=M\hat{\rho}_{k},
\]
where $\rho_{k}=Q_{\pi}^{-1}p_{n}$, $\hat{\rho}_{k}=Q_{\hat{\pi}}^{-1}\hat{\rho}_{k}$.
In particular, this makes clear why $\hat{K}$ is the natural coarse-graining
Markov matrix of $K$.
\begin{proof}
Since $\hat{K}$ is the composition of Markov matrices it is itself
a Markov matrix. Moreover, we see that $\hat{\pi}$ is the invariant
measure of $\hat{K}$, because
\[
\hat{K}^{*}\hat{\pi}=M^{*}K^{*}N^{*}\hat{\pi}=M^{*}K^{*}\pi=M^{*}\pi=\hat{\pi}\,.
\]
Considering the Markov chain, let $p_{k}=K^{*}p_{k-1}$ be given.
Assuming that $p_{k}=N^{*}\hat{p}_{k}$, we conclude that $\hat{p}_{k}=M^{*}N^{*}\hat{p}_{k}=M^{*}p_{k}$.
Hence,
\[
\hat{p}_{k}=M^{*}p_{k}=M^{*}K^{*}p_{k-1}=M^{*}K^{*}N^{*}\hat{p}_{k-1}=\hat{K}^{*}\hat{p}_{k-1}\,.
\]
Finally, if $K$ satisfies detailed balance with respect to $\pi$
then $Q_{\pi}K=K^{*}Q_{\pi}$, and hence, we have
\begin{align*}
Q_{\hat{\pi}}\hat{K} & =Q_{\hat{\pi}}NKM=M^{*}Q_{\pi}KM=M^{*}K^{*}Q_{\pi}M=M^{*}K^{*}N^{*}Q_{\hat{\pi}}=\hat{K}^{*}Q_{\hat{\pi}}\ .
\end{align*}
\end{proof}
Conversely, if $\hat{p}_{k+1}=\hat{K}^{*}\hat{p}_{k}$, and $\hat{K}=NKM$,
then a direct computation shows that $p_{k}:=N^{*}\hat{p}_{k}$ solves
the projected Markov chain $p_{k+1}=P^{*}K^{*}P^{*}p_{k}$, which
in general is different to $p_{k+1}=K^{*}p_{k}$. In this sense, the
projection $P$ describes the information loss going from a coarser
system to a finer system.
\begin{rem}
Theorem \ref{thm:Coarse-grainingMarkovMatrix} naturally generalizes
to continuous time Markov processes (see e.g. \cite{MieSte19ECLRS}).
Let $p(t)=\e^{tA^{*}}p_{0}$ or equivalently $p$ solving $\dot{p}=A^{*}p$
be given, where $A$ is a Markov generator such that $\e^{tA}$ is
a semigroup of Markov matrices. If $p(t)=N^{*}\hat{p}(t)$, then the
coarse-grained probability vectors $\hat{p}$ satisfies the coarse-grained
Markov process $\dot{\hat{p}}=\hat{A}^{*}\hat{p}$, with the coarse-grained
Markov generator
\[
\hat{A}=NAM.
\]
\end{rem}

\section{Coarse-grained network\label{sec:Coarse-grained-network}}

Graph theoretically, a Markov matrix $K$ defines a directed graph
$G=G(V,E)$, with vertices given by the state space $\mathcal{Z}$
and edges between states $z_{i}$ and $z_{j}$ whenever $K_{ij}>0$.
An equivalence relation given by the coarse-graining map $\phi:\mathcal{Z}\to\hat{\mathcal{Z}}$
(i.e. $z_{i}\sim z_{j}$ iff $\phi(z_{i})=\phi(z_{j})$), defines
a partition of the the graph into blocks. By definition, these blocks
define the vertices of the so-called \textit{quotient graph} $\hat{G}$.
The edges in the quotient graph are defined as follows: two blocks
$B_{1}$ and $B_{2}$ are adjacent if some vertex in $B_{1}$ is adjacent
to some vertex in $B_{2}$ with respect to the edges in the starting
graph. That means if on $G=G(V,E)$ there is an equivalence relation
$\sim$ then $\hat{G}$ has vertices $\hat{V}=V/\sim$ and edges $\left\{ ([u]_{\sim},[v]_{\sim}):(u,v)\in E\right\} $.
In particular, the edges in each equivalence class (or block) vanish
(see e.g. \cite{Boll98MGT}).

Since we are interested in functions that are defined on the edges
(e.g. fluxes), we translate the above state-based coarse-graining
procedure to edges. For this it is convenient to introduce tensor
spaces.

\subsection{Coarse-graining in tensor spaces}

Naturally the space over the edges can be identified by matrices or
equivalently by the tensor product space $X\otimes X\simeq L(X^{*},X)$.
In particular, we use both formulations and switch between them whenever
necessary. In principle, also multi-tensor spaces can be considered
for example to capture cycles between several states. However, we
restrict ourselves to tensors of second order.

Importantly the tensor space $X\otimes X$ is consistent with the
concept of Markov matrices capturing positivity and duality. Positivity
is again defined pointwise. The constant 1-element in $X\otimes X$
is given by $\one\otimes\one$ and will be denoted by $\one_{\otimes}$.
The dual space of $L(X^{*},X)$ is given by $L(X^{*},X)^{*}\simeq L(X^{**},X^{*})\simeq L(X,X^{*})\simeq X^{*}\otimes X^{*}$.
The duality mapping is given by
\[
A\in L(X,X^{*})\mapsto\mathrm{Tr}(A^{*}\cdot)\in L(X^{*},X)^{*}.
\]
In the following we will denote the dual paring between $L(X^{*},X)$
and $L(X,X^{*})$ by 
\[
\lla A,B\rra:=\mathrm{Tr}(A^{*}B)=\mathrm{Tr}(AB^{*})=\sum_{i,j}A_{ij}B_{ij},
\]
which is just the usual dual paring by pointwise multiplication if
the matrices are understood as $n\times n$-vectors.

We may also define multiplication operators with elements of the dual
space. For a given matrix $m\in L(X,X^{*})\simeq X^{*}\otimes X^{*}$
we define the (diagonal) multiplication operator by
\[
Q_{m}:L(X^{*},X)\rightarrow L(X,X^{*}),\quad Q_{m}b=\left(m_{ij}b_{ij}\right)_{ij}
\]
by pointwise multiplication. Clearly, we have that $Q_{m}\one_{\otimes}=m$.
Indeed, the target space of $Q_{m}$ makes sense which can been seen
from the following observation that $Q_{m}$ is symmetric:
\[
\lla c,Q_{m}b\rra=\sum_{i,j}m_{ij}b_{ij}c_{ij}=\sum_{i,j}m_{ij}c_{ij}b_{ij}=\lla b,Q_{m}c\rra.
\]
Of great importance for us is the element $m=Q_{\pi}K\in L(X,X^{*})\simeq X^{*}\otimes X^{*}$,
which can be understood as a weight function defined on the edges.

For the coarse-grained state space $\hat{X}$, we analogously define
$\hat{X}\otimes\hat{X}\simeq L(\hat{X}^{*},\hat{X})$ and $\hat{X}^{*}\otimes\hat{X}^{*}\simeq L(\hat{X},\hat{X}^{*})$
by replacing $X$ by $\hat{X}$. We define a coarse-graining operator
on $L(X^{*},X)$ by
\begin{align*}
\widetilde{M} & :L(\hat{X}^{*},\hat{X})\simeq\hat{X}\otimes\hat{X} & \to & \quad L(X^{*},X)\simeq X\otimes X\\
 & \qquad\hat{b} & \mapsto & \quad\widetilde{M}\hat{b}:=M\hat{b}M^{*}\,.
\end{align*}

\begin{prop}
\label{prop:CoarseGrainingFluxes}The operator $\widetilde{M}$ has
the following properties:
\begin{enumerate}
\item $\widetilde{M}$ is again a deterministic Markov operator.
\item The adjoint operator is given by
\[
\widetilde{M}^{*}:L(X,X^{*})\to L(\hat{X},\hat{X}^{*}),\quad\widetilde{M}^{*}b=M^{*}bM\,.
\]
\item Let $m:=Q_{\pi}K$. Then $\hat{m}:=\widetilde{M}^{*}m=Q_{\hat{\pi}}\hat{K}$.
\end{enumerate}
\end{prop}

\begin{proof}
Clearly, $\widetilde{M}$ is again positive. Moreover, it maps to
constant 1-function $\hat{\one}_{\otimes}=\hat{\one}\otimes\hat{\one}$
in $\hat{X}\otimes\hat{X}$ to the constant 1-function $\one_{\otimes}=\one\otimes\one$
in $X\otimes X$ because we have
\[
\widetilde{M}\left(\hat{\one}\otimes\hat{\one}^{*}\right)=M\left(\hat{\one}\otimes\hat{\one}\right)M^{*}=M\hat{\one}\otimes M\hat{\one}=\one\otimes\one.
\]
Hence, $\widetilde{M}$ is a Markov operator. To see that it is deterministic,
we use the representation of the adjoint operator $\widetilde{M}^{*}$
which is the second claim and proved below. Using that, we have $\widetilde{M}^{*}\left(e_{i}\otimes e_{j}\right)=M^{*}e_{i}\otimes M^{*}e_{j}$
which is again a pure state in $\hat{X}^{*}\otimes\hat{X}^{*}$

To compute the adjoint operator $\widetilde{M}^{*}$ we have for $\hat{b}\in L(\hat{X}^{*},\hat{X})$
and $c\in L(X,X^{*})$ that 
\begin{align*}
\lla\widetilde{M}\hat{b},c\rra & =\mathrm{Tr}\left(\left(M\hat{b}M^{*}\right)^{*}c\right)=\mathrm{Tr}\left(M\hat{b}^{*}M^{*}c\right)=\mathrm{Tr}\left(M^{*}cM\hat{b}^{*}\right)=\lla M^{*}cM,\hat{b}\rra,
\end{align*}
where we used that the trace is invariant under commuting matrices.
Hence, $\widetilde{M}^{*}c=M^{*}cM$.

For the last claim, we observe that $\hat{m}=\widetilde{M}^{*}m=M^{*}Q_{\pi}KM=Q_{\hat{\pi}}NKM=Q_{\hat{\pi}}\hat{K}$.
\end{proof}
Proposition \ref{prop:CoarseGrainingFluxes} shows that the (dual)
deterministic coarse-graining operator $\widetilde{M}^{*}$ maps the
weights $m=Q_{\pi}K$ on the coarse-grained weight $\hat{m}=Q_{\hat{\pi}}\hat{K}$.
This allows to define a reconstruction operator $\widetilde{N}$ as
the inverse operator of $\widetilde{M}$ with respect to $m$ as in
Section \ref{sec:Coarse-graining}. We define 
\begin{equation}
\widetilde{N}:L(X^{*},X)\to L(\hat{X}^{*},\hat{X}),\quad\widetilde{N}=Q_{\hat{m}}^{-1}\widetilde{M}^{*}Q_{m}\,.\label{eq:NTilde}
\end{equation}
Again, we have that $\widetilde{N}\one_{\otimes}=\hat{\one}_{\otimes}$.
Its adjoint (with respect to $\lla\cdot,\cdot\rra$) is given by
\[
\widetilde{N}^{*}:L(\hat{X},\hat{X}^{*})\to L(X,X^{*}),\quad\widetilde{N}^{*}=Q_{m}\widetilde{M}Q_{\hat{m}}^{-1}\,.
\]
Clearly, the operator $\widetilde{N}^{*}$ maps $\hat{m}$ to $m$.
The definition of reconstruction operator $\widetilde{N}$ has two
advantages. First it allows to define a coarse-grained incidence matrix
as we will see next. The incidence matrix will be crucial for estimating
Poincaré-type constants in Section \ref{sec:FunctiIn-PoincareConstants}.
Moreover, it can be used to reconstruct fluxes, which are functions
on edges (see Section \ref{sec:Detailed-balance-Markov-Matrices}).

\subsection{Coarse-graining of the incidence matrix}

The connection between $X$ and $X\otimes X$ is given by the incidence
matrix (or operator) $D:X\to L(X^{*},X)\approx X\otimes X$ for the
complete graph of the vertices $V$, which is (in coordinates) defined
by
\[
De_{i}=\sum_{j}\left(e_{i}\otimes e_{j}-e_{j}\otimes e_{i}\right)\in X\otimes X.
\]
We remark that this definition distinguish between outgoing and ingoing
edges.

To define the adjoint operator, we fix the canonical basis in the
dual space $e_{k}^{*}\in X^{*}$ with $\langle e_{k}^{*},e_{i}\rangle=\delta_{ik}$.
This also defines a basis $\left\{ e_{k}^{*}\otimes e_{l}^{*}\right\} _{k,l}$
in the tensor space $X^{*}\otimes X^{*}$ such that it holds $\lla e_{k}^{*}\otimes e_{l}^{*},e_{i}\otimes e_{j}\rra=\delta_{ijkl}$.
The adjoint operator $D^{*}$ is given by 
\begin{align}
D^{*} & :X^{*}\otimes X^{*}\to X^{*},\nonumber \\
D^{*}\left(e_{i}^{*}\otimes e_{j}^{*}\right)(e_{l}) & =\lla De_{l},e_{i}^{*}\otimes e_{j}^{*}\rra=\sum_{k}\lla e_{l}\otimes e_{k}-e_{k}\otimes e_{l},e_{i}^{*}\otimes e_{j}^{*}\rra\label{eq:Dstar}\\
 & =\sum_{k}\lla e_{l}\otimes e_{k}-e_{k}\otimes e_{l},e_{i}^{*}\otimes e_{j}^{*}\rra=\begin{cases}
0 & l\neq i,l\neq j\\
1 & l=i\\
-1 & l=j
\end{cases}\,,\nonumber 
\end{align}
whenever $i\neq j$, and otherwise it is zero.

The next result shows that the coarse-graining procedure is consistent
with the definition of the incidence matrix and the quotient graph.
\begin{thm}
\label{thm:CoarseGrainedIncidenceMatrix}Let $M:\hat{X}\rightarrow X$
be given as above, which is in local coordinates $M\hat{e}_{k}=\sum_{i\in\phi^{-1}(k)}e_{i}$.
Let $m\in X^{*}\otimes X^{*}$ be arbitrary and fixed, and let $\widetilde{N}$
be defined by \eqref{eq:NTilde}. The operator $\hat{D}:\hat{X}\rightarrow L(\hat{X}^{*},\hat{X})$
defined by 
\[
\hat{D}=\widetilde{N}DM:\hat{X}\to L(\hat{X}^{*},\hat{X})
\]
is an incidence matrix, i.e. we have 
\[
\hat{D}\hat{e}_{k}=\sum_{l}\left(\hat{e}_{k}\otimes\hat{e}_{l}-\hat{e}_{l}\otimes\hat{e}_{k}\right)\,.
\]
Moreover, it holds $\widetilde{M}\hat{D}=DM$.
\end{thm}

Remarkably, the form of $\hat{D}$ is independent of the $m\in X\otimes X$.
We note that the second claim does not follow immediately from $\hat{D}=\widetilde{N}DM$,
which would imply $\widetilde{M}\hat{D}=\widetilde{P}DM$ with the
projection $\widetilde{P}=\widetilde{M}\widetilde{N}$ on $X\otimes X$.
The relation $\widetilde{M}\hat{D}=DM$ is finer and provides that
the coarse-graining procedure is consistent with the definition of
the quotient graph.
\begin{proof}
For the proof, we compute $\hat{D}\hat{e}_{k}$ explicitly. We have
$M\hat{e}_{k}=\sum_{i\in\phi^{-1}(k)}e_{i}$ and hence, 
\[
DM\hat{e}_{k}=\sum_{j}M\hat{e}_{k}\otimes e_{j}-e_{j}\otimes M\hat{e}_{k}=\sum_{j}\sum_{i\in\phi^{-1}(k)}e_{i}\otimes e_{j}-e_{j}\otimes e_{i}\,.
\]
Moreover, we have that $\widetilde{N}=Q_{\hat{m}}^{-1}\widetilde{M}^{*}Q_{m}$.
To evaluate $\widetilde{N}\left(e_{i}\otimes e_{j}\right)$, we first
observe that $\widetilde{M}^{*}\left(e_{i}\otimes e_{j}\right)=\hat{e}_{\phi(i)}\otimes\hat{e}_{\phi(j)}$
which implies that 
\begin{align*}
\widetilde{M}^{*}Q_{m}DM\hat{e}_{k} & =\sum_{j}\sum_{i\in\phi^{-1}(k)}m_{ij}\widetilde{M}^{*}\left(e_{i}\otimes e_{j}\right)-m_{ji}\widetilde{M}^{*}\left(e_{j}\otimes e_{i}\right)\\
 & =\sum_{j}\sum_{i\in\phi^{-1}(k)}m_{ij}\hat{e}_{\phi(i)}\otimes\hat{e}_{\phi(j)}-m_{ji}\hat{e}_{\phi(j)}\otimes\hat{e}_{\phi(i)}\\
 & =\sum_{l_{j}}\sum_{j\in\phi^{-1}(l_{j})}\sum_{i\in\phi^{-1}(k)}m_{ij}\hat{e}_{\phi(i)}\otimes\hat{e}_{\phi(j)}-m_{ji}\hat{e}_{\phi(j)}\otimes\hat{e}_{\phi(i)}\\
 & =\sum_{l_{j}}\hat{m}_{kl_{j}}\hat{e}_{k}\otimes\hat{e}_{l_{j}}-\hat{m}_{l_{j}k}\hat{e}_{l_{j}}\otimes\hat{e}_{k},
\end{align*}
where we have used the definition of $\hat{m}=\widetilde{M}^{*}m$
given by $\hat{m}_{kl}=\sum_{j\in\phi^{-1}(l)}\sum_{i\in\phi^{-1}(k)}m_{ij}$.
Hence, we conclude that 
\[
\hat{D}\hat{e}_{k}=\widetilde{N}DM\hat{e}_{k}=Q_{\hat{m}}^{-1}\widetilde{M}^{*}Q_{m}DM\hat{e}_{k}=\sum_{l}\hat{e}_{k}\otimes\hat{e}_{l}-\hat{e}_{l}\otimes\hat{e}_{k}\ ,
\]
which is the desired formula.

Now we prove $\widetilde{M}\hat{D}=DM$ again by direct calculation.
We have $\widetilde{M}\left(\hat{e}_{k}\otimes\hat{e}_{l}\right)=M\hat{e}_{k}\otimes M\hat{e}_{l}$
and hence,
\[
\widetilde{M}\hat{D}\hat{e}_{k}=\sum_{l}M\hat{e}_{k}\otimes M\hat{e}_{l}-M\hat{e}_{l}\otimes M\hat{e}_{k}=\sum_{l}\sum_{j\in\phi^{-1}(l)}M\hat{e}_{k}\otimes e_{j}-e_{j}\otimes M\hat{e}_{k}=DM\hat{e}_{k},
\]
which we wanted to show.
\end{proof}
\begin{rem}
We finally remark that the coarse-graining procedure can also be applied
to undirected graphs. Introducing the space $X\odot X\simeq L_{\mathrm{sym}}(X^{*},X)$
containing the symmetric tensors (or matrices), the coarse-graining
operator $\widetilde{M}$ respect the symmetric structure. Moreover,
if $m=Q_{\pi}K$ is symmetric then also $\widetilde{N}$ maps into
symmetric tensors. In the following we will treat the case of symmetric
$m$ with more detail but we will not consider undirected graphs.
\end{rem}

\section{Detailed balance Markov matrices\label{sec:Detailed-balance-Markov-Matrices}}

A special situation occurs if the operator $K$ satisfies detailed
balance meaning that $m_{ij}=\pi_{i}K_{ij}=\pi_{j}K_{ji}=m_{ji}$.
We define the associated Markov generator $A$ by $A=K-\mathrm{id}$
and investigate the following evolution system in $X^{*}$:
\[
\dot{c}=A^{*}c\,.
\]
In the next lemma, we recall that assuming that $K$ (or equivalently
$A$) satisfies detailed balance, the system $\dot{c}=A^{*}c$ can
be written as a gradient flow expressed via the incidence operator
\begin{align*}
\dot{c} & =-D^{*}b,\\
b & =\frac{1}{2}Q_{m}D\rho,\\
\rho & =Q_{\pi}^{-1}c\ ,
\end{align*}
where the first equation is a continuity equation between the fluxes
$b$ and the concentrations $c$, the last equation defines the relative
densities $\rho$ of $c$ with respect to $\pi$ and the second equation
is the constitutive relation between the relative densities $\rho$
and the fluxes $b$, which uses the tensor valued diagonal operator
$Q_{m}$. Note that there is a factor $\tfrac{1}{2}$ because the
incidence operator $D$ counts every edge twice.
\begin{lem}
Let $K$ (or equivalently $A$) satisfy the detailed balance condition.
Then $A^{*}=-\frac{1}{2}D^{*}Q_{m}DQ_{\pi}^{-1}$.
\end{lem}

\begin{proof}
By direct computation, we have for $\rho=\sum_{i}\rho_{i}e_{i}$ that
\[
D\rho=\sum_{i}\rho_{i}De_{i}=\sum_{i,j}\rho_{i}\left(e_{i}\otimes e_{j}-e_{j}\otimes e_{i}\right)\,.
\]
Hence, we get that $Q_{m}D\rho=\sum_{i,j}\rho_{i}m_{ij}e_{i}^{*}\otimes e_{j}^{*}-\rho_{i}m_{ji}e_{j}^{*}\otimes e_{i}^{*}$,
which implies
\[
-D^{*}Q_{m}D\rho=-\sum_{i,j}\rho_{i}m_{ij}D^{*}\left(e_{i}^{*}\otimes e_{j}^{*}\right)+\rho_{i}m_{ji}D^{*}\left(e_{j}^{*}\otimes e_{i}^{*}\right)\,.
\]
Evaluating both sides at $e_{l}$ and using the explicit formula \eqref{eq:Dstar},
we get that
\begin{align*}
\left(-D^{*}Q_{m}D\rho\right)e_{l} & =-\sum_{i,j}\rho_{i}m_{ij}D^{*}\left(e_{i}^{*}\otimes e_{j}^{*}\right)e_{l}+\sum_{i,j}\rho_{i}m_{ji}D^{*}\left(e_{j}^{*}\otimes e_{i}^{*}\right)e_{l}\\
 & =-\sum_{j}\rho_{l}m_{lj}+\sum_{i}\rho_{i}m_{il}+\sum_{i}\rho_{i}m_{li}-\sum_{j}\rho_{l}m_{jl}\\
 & =-\rho_{l}\sum_{j}\left(m_{lj}+m_{jl}\right)+\sum_{i}\rho_{i}\left(m_{il}+m_{li}\right)\,.
\end{align*}
Using that $m_{il}=m_{li}=\pi_{i}A_{il}=\pi_{l}A_{li}$ and $\rho_{i}=c_{i}/\pi_{i}$,
we get
\[
\left(-\frac{1}{2}D^{*}Q_{m}D\rho\right)e_{l}=-c_{l}\sum_{j}m_{lj}/\pi_{l}+\sum_{i}c_{i}m_{il}/\pi_{i}=-c_{l}\sum_{j}A_{lj}+\sum_{i}c_{i}A_{il},
\]
which implies that $A^{*}c=-\frac{1}{2}D^{*}Q_{m}DQ_{\pi}^{-1}c$
for all $c\in X$.
\end{proof}

\subsection{Coarse-graining for detailed balance Markov operators}

We are going to show that the above gradient flow decomposition is
also consistent with the coarse-graining procedure. Crucial for that
result is the relation $DM=\widetilde{M}\hat{D}$.
\begin{thm}
\label{thm:CgDetBalMarkovProcess}Let us assume that there is an equilibration
of the concentrations $c=N^{*}\hat{c}$. Then the coarse-grained concentrations
$\hat{c}$ solves the coarse-grained evolution equation of the form
\begin{align*}
\dot{\hat{c}} & =-\hat{D}^{*}\hat{b},\\
\hat{b} & =\frac{1}{2}Q_{\hat{m}}\hat{D}\hat{\rho},\\
\hat{\rho} & =Q_{\hat{\pi}}^{-1}\hat{c}\ .
\end{align*}
In particular, we have an equilibration of the fluxes $b=\widetilde{N}^{*}\hat{b}$.
\end{thm}

\begin{proof}
Since $M^{*}N^{*}=\mathrm{id}_{X^{*}}$, we observe that $M^{*}c=\hat{c}$.
By Theorem \ref{thm:CoarseGrainedIncidenceMatrix} we have that $DM=\widetilde{M}\hat{D}$.
Hence, we obtain
\[
\dot{\hat{c}}=M^{*}\dot{c}=-M^{*}D^{*}b=-\hat{D}^{*}\widetilde{M}^{*}b,
\]
Defining $\hat{b}:=\widetilde{M}^{*}b$ (which implies that $b=\widetilde{N}^{*}\hat{b}$)
and using that $\widetilde{M}^{*}Q_{m}=Q_{\hat{m}}\widetilde{N}$
and $Q_{\pi}^{-1}N^{*}=MQ_{\hat{\pi}}^{-1}$, we get
\[
\hat{b}=\widetilde{M}^{*}b=\frac{1}{2}\widetilde{M}^{*}Q_{m}DQ_{\pi}^{-1}N^{*}\hat{c}=\frac{1}{2}Q_{\hat{m}}\widetilde{N}DMQ_{\hat{\pi}}^{-1}\hat{c}=\frac{1}{2}Q_{\hat{m}}\hat{D}\hat{\rho},
\]
where we have introduced the coarse-grained relative density $\hat{\rho}=Q_{\hat{\pi}}^{-1}\hat{c}$.
\end{proof}

\subsection{Flux reconstruction\label{subsec:Flux-reconstruction}}

Theorem \ref{thm:CgDetBalMarkovProcess} provides that an equilibration
of the fluxes necessarily occurs if concentrations equilibrate. In
practice, often the converse question arises, namely how fluxes on
the large graph can be reconstructed out of the coarse-grained quantities
like concentrations and fluxes. As one would guess $b=\widetilde{N}^{*}\hat{b}$
is not the desired flux because additional fluxes between coarse-grained
states are needed.

The next proposition provides an affirmative answer and shows that
for a given pair $(\hat{c},\hat{b})$ satisfying the continuity equation
$\dot{\hat{c}}=-\hat{D}^{*}\hat{b}$, then there is a reconstructed
flux $b$ such that $c=N^{*}\hat{c}$ solves $\dot{c}=-D^{*}b$ and
the reconstructed flux consists of two parts $b=b_{1}+b_{2}$ such
that $b_{1}=\widetilde{N}^{*}\hat{b}$ and $b_{2}\in\mathrm{Ker}\hat{D}^{*}\widetilde{M}^{*}$.
Moreover, $b$ depends linearly on $\hat{b}$.
\begin{prop}
Let $(\hat{c},\hat{b})$ satisfying the continuity equation $\dot{\hat{c}}=-\hat{D}^{*}\hat{b}$.
Define the reconstructed concentrations by $c=N^{*}\hat{c}$. Then
there is a reconstructed flux $b$, which solves the macroscopic continuity
equation $\dot{c}=-D^{*}b$ and $b$ is given by $b=b_{1}+b_{2}$
with $b_{1}=\widetilde{N}^{*}\hat{b}$ and $b_{2}\in\mathrm{Ker}\hat{D}^{*}\widetilde{M}^{*}$,
which depends linear on $\hat{b}$.
\end{prop}

\begin{proof}
Let $b_{1}=\widetilde{N}^{*}\hat{b}$. Then we have to construct $b_{2}\in X^{*}\otimes X^{*}$
such that 
\[
\dot{c}=-D^{*}(b_{1}+b_{2})=-D^{*}(\widetilde{N}^{*}\hat{b}+b_{2}).
\]
Since $c=N^{*}\hat{c}$, which implies $\dot{c}=N^{*}\dot{\hat{c}}=-N^{*}\hat{D}^{*}\hat{b}$,
we get that $b_{2}$ has to satisfy the linear equation 
\begin{equation}
D^{*}b_{2}=\left(N^{*}\hat{D}^{*}-D^{*}\widetilde{N}^{*}\right)\hat{b}=:x^{*}\in X^{*}.\label{eq:Dstarb2=00003Dxstar}
\end{equation}
Using Fredholm's alternative, this equation is solvable if for all
$x\in\mathrm{Ker}(D)$ we have that $\langle x,x^{*}\rangle=0$.

So let $x\in\mathrm{Ker}(D)\subset X$. Hence, 
\[
\langle x,x^{*}\rangle=\langle x,\left(N^{*}\hat{D}^{*}-D^{*}\widetilde{N}^{*}\right)\hat{b}\rangle=\langle\hat{D}Nx,\hat{b}\rangle\,.
\]
Using Proposition \ref{prop:CoarseGrainingProperties}, we decompose
$x\in X=\mathrm{Ker}(N)+\mathrm{Range}(M)$. Hence, for proving that
$\langle x,x^{*}\rangle=0$, we may assume that $x\in\mathrm{Range}(M)$,
i.e. $M\hat{x}=x$. Using that $\hat{b}=\widetilde{M}^{*}b_{1}$ and
that $NM=\mathrm{id}_{\hat{X}}$, we compute
\[
\langle x,x^{*}\rangle=\langle\hat{D}NM\hat{x},\widetilde{M}^{*}b_{1}\rangle=\langle\widetilde{M}\hat{D}\hat{x},b_{1}\rangle=\langle DM\hat{x},b_{1}\rangle=\langle Dx,b_{1}\rangle=0,
\]
where we have used that $\widetilde{M}\hat{D}=DM$. Hence, there exists
$b_{2}$ such that $D^{*}b_{2}=\left(N^{*}\hat{D}^{*}-D^{*}\widetilde{N}^{*}\right)\hat{b}$,
which implies that $b=b_{1}+b_{2}$ solves $\dot{c}=-D^{*}b$. By
\eqref{eq:Dstarb2=00003Dxstar} we see that $b_{2}$ depends linearly
on $\hat{b}$.

To see that $b_{2}\in\mathrm{Ker}(\hat{D}^{*}\widetilde{M}^{*})$,
we observe that
\[
\hat{D}^{*}\widetilde{M}^{*}b_{2}=M^{*}D^{*}b_{2}=M^{*}\left(N^{*}\hat{D}^{*}-D^{*}\widetilde{N}^{*}\right)\hat{b}=\left(M^{*}N^{*}M^{*}D^{*}\widetilde{N}^{*}-M^{*}D^{*}\widetilde{N}^{*}\right)\hat{b}=0.
\]
\end{proof}
\begin{rem}
We note that the existence of $b_{2}$ as well as the linear dependence
on $\hat{b}$ as been used in \cite{Step20EDPCLRDS} to coarse-grain
fast-slow linear reaction-diffusion systems.%
\end{rem}

\section{Functional inequalities and Poincaré constants\label{sec:FunctiIn-PoincareConstants}}

In this section we apply the coarse-graining procedure to derive estimates
between functionals on $X$ and $\hat{X}$. First we observe the following
for the expectations
\[
\mathbb{E}_{\pi}(x):=\sum_{i\in\mathcal{Z}}\pi_{i}x_{i}=\langle x,\pi\rangle,\quad\mathbb{E}_{\hat{\pi}}(\hat{x}):=\sum_{j\in\hat{\mathcal{Z}}}\hat{\pi}_{j}\hat{x}_{j}=\langle\hat{x},\hat{\pi}\rangle\,.
\]

\begin{lem}
\label{lem:Expectation}If $\hat{x}=Nx$ or $x=M\hat{x}$, then $\mathbb{E}_{\pi}(x)=\mathbb{E}_{\hat{\pi}}(\hat{x})$.
\end{lem}

\begin{proof}
If $\hat{x}=Nx$, we have $\mathbb{E}_{\hat{\pi}}(\hat{x})=\langle\hat{\pi},\hat{x}\rangle=\langle\hat{\pi},Nx\rangle=\langle N^{*}\hat{\pi},x\rangle=\langle\pi,x\rangle=\mathbb{E}_{\pi}(x)$.
If $x=M\hat{x}$ we conclude that $Nx=\hat{x}$.
\end{proof}

\subsection{Coarse-graining for energy functionals}

For any strictly convex and non-negative function $\Phi:\R\rightarrow[0,\infty[$
we define the associated energy functional on $X$ by 
\[
\mathcal{E}_{\Phi}(x)=\mathbb{E}_{\pi}\Phi(x)-\Phi(\mathbb{E}_{\pi}x)=\langle\Phi(x),\pi\rangle-\Phi(\langle x,\pi\rangle),
\]
where for a vector $x\in X\simeq\R^{n}$ the function $\Phi(x)\in X$
is defined componentwise, i.e. $\Phi(x)_{i}=\Phi(x_{i})$. Note, that
the measure $\pi\in X^{*}$ is fixed and implicitly given in the definition
of $\mathcal{E}_{\Phi}$.
\begin{lem}
\label{lem:PropertiesEnergyFunctional}The functional $\mathcal{E}_{\Phi}$
is non-negative and its minimum is attained on constant vectors. Moreover,
we have for the rescaled function $\widetilde{\Phi}(r)=\Phi(r)+cr+d$
that $\cE_{\widetilde{\Phi}}=\cE_{\Phi}$.
\end{lem}

\begin{proof}
The first claim follows directly be Jensen's inequality since $\Phi$
is strictly convex. Moreover, a direct computation shows that
\begin{align*}
\cE_{\widetilde{\Phi}}(x) & =\mathbb{E}_{\pi}\widetilde{\Phi}(x)-\widetilde{\Phi}(\mathbb{E}_{\pi}x)=\langle\Phi(x)+cx+d,\pi\rangle-\Phi(\mathbb{E}_{\pi}x)-c\,\mathbb{E}_{\pi}x-d=\\
 & =\langle\Phi(x),\pi\rangle+c\langle x,\pi\rangle+d\langle\one,\pi\rangle-\Phi(\mathbb{E}_{\pi}x)-c\,\mathbb{E}_{\pi}x-d=\cE_{\Phi}(x)\,.
\end{align*}
\end{proof}
Typical examples for the function $\Phi$ are:
\begin{enumerate}
\item $\Phi(r)=\frac{1}{2}r^{2}$. Then $\mathcal{E}_{\Phi}(x)=\frac{1}{2}\left(\langle x^{2},\pi\rangle-\langle x,\pi\rangle^{2}\right)$
corresponds to the quadratic energy or statistical variance.
\item $\Phi(r)=r\log r-r+1$. Then $\cE_{\Phi}$ corresponds to the free
energy of Boltzmann type, which will be denoted by $\mathrm{Ent}_{\pi}$
in the following.
\end{enumerate}
\begin{rem}
There are several remarks in order.
\begin{enumerate}
\item Often functionals of the form $\cE^{*}(p)=\langle\Psi(p/\pi),\pi\rangle$
as relative energies or entropies are considered. In contrast to $\cE_{\Phi}$,
which is defined on $X$, functionals of the latter form are defined
on probability vectors as elements of the dual space $X^{*}$. However,
they are related via the Legendre transform. To see this, forgetting
about the normalization term $-\Phi(\mathbb{E}_{\pi}x)$, the Legendre
transform of the functional $\widetilde{\cE_{\Phi}}$, $\widetilde{\cE_{\Phi}}(x):=\langle\Phi(x),\pi\rangle$
is given by 
\[
\widetilde{\cE_{\Phi}}^{*}(p)=\sup_{x\in X}\left(\langle p,x\rangle-\widetilde{\cE_{\Phi}}(x)\right)=\sup_{x\in X}\left(\langle p,x\rangle-\langle\Phi(x),\pi\rangle\right).
\]
Introducing the relative density $g$ of $p$ with respect to the
positive probability vector $\pi$, we get that 
\begin{align*}
\widetilde{\cE_{\Phi}}^{*}(p) & =\sup_{x\in X}\left(\langle g\pi,x\rangle-\langle\Phi(x),\pi\rangle\right)=\sup_{x\in X}\left(\langle g\cdot x,\pi\rangle-\langle\Phi(x),\pi\rangle\right)\\
 & =\sup_{x\in X}\langle g\cdot x-\Phi(x),\pi\rangle=\langle\sup_{x\in X}\left(g\cdot x-\Phi(x)\right),\pi\rangle=\langle\Phi^{*}(g),\pi\rangle=\langle\Phi^{*}\left(\frac{p}{\pi}\right),\pi\rangle,
\end{align*}
which is exactly the desired form.
\item We could also investigate functionals of the form $\widetilde{\cE}_{\Phi}(x)=\langle\Phi(x-\mathbb{E}_{\pi}x),\pi\rangle$
as a generalization of the variance $\Phi(r)=r^{2}$. These functionals
have the property that they are always convex. However, we will restrict
to the above form.
\end{enumerate}
\end{rem}

Analogously, we define $\hat{\cE}_{\Phi}$ on $\hat{X}$ by replacing
$x$ by $\hat{x}$ and $\pi$ by $\hat{\pi}$. The functionals $\cE_{\Phi}$
on $X$ and $\hat{\cE}_{\Phi}$ on $\hat{X}$ can be estimated as
follows.
\begin{prop}
\label{prop:EstimateEnergy}We have the following relation for the
functionals regarding coarse-graining and reconstruction.
\end{prop}

\begin{enumerate}
\item For all $\hat{x}\in\hat{X}$, $x=M\hat{x}$ implies $\cE_{\Phi}(x)=\hat{\cE}_{\Phi}(\hat{x})$.
This holds even for all functions $\Phi:\R\to\R$ not necessarily
convex.
\item For all $x\in X$, $\hat{x}=Nx$ implies $\hat{\cE}_{\Phi}(\hat{x})\leq\cE_{\Phi}(x)$.
\end{enumerate}
\begin{proof}
For the first claim, take any $\hat{x}\in\hat{X}$. Then we have with
Lemma \ref{lem:Expectation} that
\begin{align*}
\cE_{\Phi}(x)=\cE_{\Phi}(M\hat{x}) & =\sum_{i}\pi_{i}\Phi(\left(M\hat{x}\right)_{i})-\Phi(\mathbb{E}_{\pi}(M\hat{x}))=\sum_{j\in\hat{Z}}\sum_{i=\phi^{-1}(j)}\pi_{i}\Phi(\left(M\hat{x}\right)_{i})-\mathbb{E}_{\hat{\pi}}\hat{x}\\
 & =\sum_{j\in\hat{Z}}\sum_{i=\phi^{-1}(j)}\pi_{i}\Phi(\hat{x}_{\phi(i)})-\mathbb{E}_{\hat{\pi}}\hat{x}=\sum_{j\in\hat{Z}}\Phi(\hat{x}_{j})\sum_{i=\phi^{-1}(j)}\pi_{i}-\mathbb{E}_{\hat{\pi}}\hat{x}\\
 & =\sum_{j\in\hat{Z}}\hat{\pi}_{j}\Phi(\hat{x}_{j})-\mathbb{E}_{\hat{\pi}}\hat{x}=\hat{\cE}_{\Phi}(\hat{x}),
\end{align*}
where we used that $\hat{\pi}_{j}=\sum_{i=\phi^{-1}(j)}\pi_{i}$.

For the second claim, take any $x\in X$. Using Jensen's inequality
for the convex function $\Phi$, which means that we have the pointwise
inequality $\Phi(Nx)\leq N\Phi(x)$, we obtain
\begin{align*}
\hat{\cE}_{\Phi}(Nx) & =\langle\hat{\pi},\Phi(Nx)\rangle-\mathbb{E}_{\hat{\pi}}(Nx)\leq\langle\hat{\pi},N\Phi(x)\rangle-\mathbb{E}_{\pi}(x)=\\
 & =\langle N^{*}\hat{\pi},\Phi(x)\rangle-\mathbb{E}_{\pi}(x)=\langle\pi,\Phi(x)\rangle-\mathbb{E}_{\pi}(x)=\cE_{\pi}(x)\,.
\end{align*}
\end{proof}
The log-Sobolev constant is defined by estimating $\cE_{\Phi}(x^{2})$
where $\Phi(r)=r\log r$, i.e. $\mathrm{Ent}_{\pi}(x^{2})$. From
the above proposition it is not clear that is possible to obtain estimates
between $\mathrm{Ent}_{\hat{\pi}}(\left(\hat{x}\right)^{2})$ and
$\mathrm{Ent}_{\pi}(\left(M\hat{x}\right)^{2})$. In fact, we prove
that this is possible even for general convex functions not necessarily
quadratic.
\begin{prop}
\label{prop:EstimateEntropy}Let $g:\R\to[0,\infty[$ be convex and
satisfy $g(x)>0$ if $x\neq0$ and $g(0)=0$. Then we have 
\[
\forall\hat{x}\in\hat{X}\quad:\quad\mathrm{Ent}_{\hat{\pi}}(g(\hat{x}))\leq\mathrm{Ent}_{\pi}(g(M\hat{x}))\,,
\]
where again with a small abuse of notation $g(x)$ is meant component-wise,
i.e. $g(x)_{i}=g(x_{i})$.
\end{prop}

\begin{proof}
The proof is done in two steps. First, we shift the function $\Phi(r)=r\log r$
to incorporate $g$. Secondly, we derive the estimate.
\begin{description}
\item [{1.$\,$Step:}] Clearly, we have equality for $\hat{x}=0$. So let
us take $\hat{x}\neq0$. Then, $g(\hat{x})>0$ and because $\pi>0$,
there is a constant $C>0$ such that $\langle g(M\hat{x}),\pi\rangle\geq C>0$.
Let us define $c>0$ by $c:=\e^{-(1+C)}$. We define $\widetilde{\Phi}(r)=r\log r+c\,r$,
which has its minimum at $r=C$. So we have $\widetilde{\Phi}(r_{1})\geq\widetilde{\Phi}(r_{2})$
for $r_{1}\geq r_{2}\geq C$. Recalling that the energy functional
is invariant under affine shifts (Lemma \ref{lem:PropertiesEnergyFunctional}),
we have that $\mathrm{Ent}_{\pi}(x)=\cE_{\widetilde{\Phi}}(x)$ and
analogously also for the coarse-grained states $\hat{x}$. So we are
going to show $\hat{\cE}_{\widetilde{\Phi}}(g(\hat{x}))\leq\cE_{\widetilde{\Phi}}(g(M\hat{x}))$,
or, equivalently,
\[
\langle\widetilde{\Phi}(g(\hat{x})),\hat{\pi}\rangle-\widetilde{\Phi}(\langle g(\hat{x}),\hat{\pi}\rangle)\leq\langle\widetilde{\Phi}(g(M\hat{x})),\pi\rangle-\widetilde{\Phi}(\langle g(M\hat{x}),\pi\rangle)\,.
\]
\item [{2.$\,$Step:}] We observe that (completely similar to the proof
of Proposition \ref{prop:EstimateEnergy}) 
\begin{align*}
\langle\widetilde{\Phi}(g(M\hat{x})),\pi\rangle & =\sum_{i}\widetilde{\Phi}(g(M\hat{x}))_{i}\pi_{i}=\sum_{i}\widetilde{\Phi}(g(M\hat{x})_{i})\pi_{i}=\sum_{j\in\hat{Z}}\sum_{i=\phi^{-1}(j)}\pi_{i}\widetilde{\Phi}\circ g(\hat{x}_{\phi(i)})\\
 & =\sum_{j\in\hat{Z}}\hat{\pi}_{j}\widetilde{\Phi}\circ g(\hat{x}_{j})=\langle\widetilde{\Phi}(g(\hat{x})),\hat{\pi}\rangle\,.
\end{align*}
Hence, it suffices to prove that $\widetilde{\Phi}(\langle g(\hat{x}),\hat{\pi}\rangle)\geq\widetilde{\Phi}(\langle g(M\hat{x}),\pi\rangle)$.
To see this we use Jensen's inequality which states that $g(M\hat{x})\leq Mg(\hat{x})$
for all Markov matrices $M$ and convex functions $g$. Hence, we
get $\langle g(M\hat{x}),\pi\rangle\leq\langle Mg(\hat{x}),\pi\rangle=\langle g(\hat{x}),M^{*}\pi\rangle=\langle g(\hat{x}),\hat{\pi}\rangle$.
Since $\langle g(M\hat{x}),\pi\rangle\geq C$ and the function $\widetilde{\Phi}$
is monotone for arguments larger than $C$ by construction, we conclude
that also $\widetilde{\Phi}\left(\langle g(M\hat{x}),\pi\rangle\right)\leq\widetilde{\Phi}\left(\langle g(\hat{x}),\hat{\pi}\rangle\right)$.
Hence, the claim is proved.
\end{description}
\end{proof}

\subsection{Dirichlet forms and Poincaré-type estimates}

To estimate Poincaré-type constants, we introduce the Dirichlet form
(or dissipation) for $K$ and $\hat{K}$ by
\[
\cD_{K}(x)=\frac{1}{2}\sum_{i,j}\pi_{i}K_{ij}(x_{i}-x_{j})^{2},\quad\cD_{\hat{K}}(x)=\frac{1}{2}\sum_{i,j}\hat{\pi}_{i}\hat{K}_{ij}(x_{i}-x_{j})^{2}\,.
\]

Without loss of generality, we assume that $K$ and $\hat{K}$ satisfy
detailed balance, because the Dirichlet form takes into account only
the symmetric part of $Q_{m}K$. Using $m=Q_{\pi}K$ and $\hat{m}=Q_{\hat{\pi}}\hat{K}$,
the Dirichlet form is related to the generator $A=K-\mathrm{id}$
by 
\[
\cD_{K}(x)=\frac{1}{2}\lla Dx,Q_{m}Dx\rra=\frac{1}{2}\langle x,D^{*}Q_{m}Dx\rangle=-\langle x,A^{*}Q_{\pi}x\rangle=-\langle Ax,Q_{\pi}x\rangle=-\langle x\cdot Ax,\pi\rangle\,.
\]
Moreover, we have $\cD_{\hat{K}}(\hat{x})=\tfrac{1}{2}\lla\hat{D}\hat{x},Q_{\hat{m}}\hat{D}\hat{x}\rra=-\langle\hat{x}\cdot\hat{A}\hat{x},\hat{\pi}\rangle$.

We are interested in estimating the spectral gap $\lambda=\lambda(K,\Phi)$,
which is defined by the largest constant $c>0$ that satisfies the
discrete Poincaré-type inequality
\[
\cD_{K}(x)\geq c\,\cE_{\Phi}(x),\quad\mathrm{i.e.}\quad\lambda(K,\Phi)=\inf\left\{ \frac{\mathcal{D}_{K}(x)}{\cE_{\Phi}(x)}:\forall x\ \cE_{\Phi}(x)\neq0\right\} \,.
\]
Analogously, we define $\hat{\lambda}=\hat{\lambda}(\hat{K},\Phi)=\inf\left\{ \frac{\mathcal{D}_{\hat{K}}(\hat{x})}{\hat{\cE}_{\Phi}(\hat{x})}:\forall\hat{x}\ \hat{\cE}_{\Phi}(\hat{x})\neq0\right\} $.
We also define the log-Sobolev constants $\lambda_{g,\mathrm{LS}}=\inf\left\{ \frac{\cD_{K}(x)}{\mathrm{Ent}_{\pi}(g(x))}:\mathrm{Ent}_{\pi}(g(x))\neq0\right\} $
and analogously $\hat{\lambda}_{g,\mathrm{LS}}$.
\begin{thm}
With the above notation, we have the following:
\begin{enumerate}
\item For all $\hat{x}\in\hat{X}$ we have that $\cD_{K}(M\hat{x})=\cD_{\hat{K}}(\hat{x})$.
\item For all functions $\Phi$ we have that $\lambda(K,\Phi)\leq\lambda(\hat{K},\Phi)$.
\item We have for the log-Sobolev constants that $\lambda_{g,LS}\leq\hat{\lambda}_{g,\mathrm{LS}}$.
\end{enumerate}
\end{thm}

\begin{proof}
We have 
\begin{align*}
\cD_{K}(M\hat{x}) & =\frac{1}{2}\lla DM\hat{x},Q_{m}DM\hat{x}\rra=-\langle M\hat{x},A^{*}Q_{\pi}M\hat{x}\rangle\\
 & =-\langle\hat{x},M^{*}A^{*}N^{*}Q_{\hat{\pi}}\hat{x}\rangle=-\langle\hat{x},\hat{A}^{*}Q_{\hat{\pi}}\hat{x}\rangle=\frac{1}{2}\langle\hat{x},\hat{D}^{*}Q_{\hat{m}}D\hat{x}\rangle=\cD_{\hat{K}}(\hat{x})\,,
\end{align*}
which is the first claim. Hence, we obtain the following relations
for the functional inequalities
\begin{align*}
\lambda & =\inf\left\{ \frac{\cD_{K}(x)}{\cE_{\Phi}(x)}:\cE_{\Phi}(x)\neq0\right\} \leq\inf\left\{ \frac{\cD_{K}(M\hat{x})}{\cE_{\Phi}(M\hat{x})}:\cE_{\Phi}(M\hat{x})\neq0\right\} \\
 & =\inf\left\{ \frac{\cD_{\hat{K}}(\hat{x})}{\cE_{\Phi}(M\hat{x})}:\cE_{\Phi}(M\hat{x})\neq0\right\} \,.
\end{align*}
For the second claim, we use Proposition \ref{prop:EstimateEnergy}
which implies that $\cE_{\Phi}(M\hat{x})=\hat{\cE}_{\Phi}(\hat{x})$
and hence, $\lambda\leq\hat{\lambda}$. For the third claim we use
$\mathrm{Ent}_{\hat{\pi}}(g(\hat{x}))\leq\mathrm{Ent}_{\pi}(g(M\hat{x}))$
by Proposition \ref{prop:EstimateEntropy} to obtain the bound $\lambda_{g,LS}\leq\hat{\lambda}_{g,\mathrm{LS}}$.
\end{proof}
Proposition \ref{prop:EstimateEnergy} states that we have $\cE_{\pi}(x)\geq\cE_{\hat{\pi}}(Nx)$
for all $x\in X$. So naturally the question arises whether it is
possible to obtain uniform estimates between $\cD_{K}(x)$ and $\cD_{\hat{K}}(Nx)$.
The counterexample in the next section shows that this is in general
not true.

\subsection{Counterexample}

We compute and compare $\cD_{\hat{K}}(Nx)$ and $\cD_{K}(x)$ for
a fixed $x\in X$. We have
\begin{align*}
\cD_{\hat{K}}(Nx) & =\frac{1}{2}\lla\hat{D}Nx,Q_{\hat{m}}\hat{D}Nx\rra=\frac{1}{2}\langle x,N^{*}\hat{D}^{*}Q_{\hat{m}}\hat{D}Nx\rangle=-\langle x,P^{*}A^{*}Q_{\pi}Px\rangle=\\
 & =-\langle Px,Q_{\pi}APx\rangle=-\langle Px\cdot APx,\pi\rangle\,.
\end{align*}
Recall, that we have $\cD_{K}(x)=-\langle x\cdot Ax,\pi\rangle$.
It is clear that $\cD_{K}(x)\geq\cD_{\hat{K}}(Nx)$ holds for all
$x\in\mathrm{Range}(P)$ (because then the inequality is a trivial
equality) and also for all $x\in\mathrm{Range}(\mathrm{id}-P)=\mathrm{Ker}(P)$
(because $\cD_{\hat{K}}(Nx)=0$). In particular, we always have $\cD_{K}(x)\geq\cD_{\hat{K}}(Nx)$
in the simple case of $Z=\left\{ 1,2\right\} $.

For $a\geq0$, we define on $\R^{3}$ the parameter dependent Markov
generator
\[
A_{a}=\begin{pmatrix}-8 & 4 & 4\\
1 & -2 & 1\\
a & a & -2a
\end{pmatrix}\,.
\]
Then $A_{a}^{*}$ has the stationary measure $\pi_{a}=\frac{1}{5a+4}(a,4a,4)^{\mathrm{T}}$.
One easily checks that $A_{a}$ satisfies detailed balance with respect
to $\pi_{a}$. As in the example from Section \ref{subsec:Example},
we define the coarse-graining function $\phi:Z\rightarrow\hat{Z}$
with $\phi(1)=\hat{1}$ and $\phi(2)=\phi(3)=\hat{2}$. The corresponding
Markov operator $M:\hat{X}\rightarrow X$ is given by $M=\begin{pmatrix}1\\
 & 1\\
 & 1
\end{pmatrix}$. The coarse-grained stationary measure is given by $\hat{\pi}_{a}=M^{*}\pi_{a}=\frac{1}{5a+4}(a,4a+4)^{\mathrm{T}}\,.$The
inverse operator $N_{a}:X\rightarrow\hat{X}$ and the projection $P_{a}:X\to X$
are given by
\[
N_{a}=Q_{\hat{\pi}_{a}}^{-1}M^{*}Q_{\pi_{a}}=\begin{pmatrix}1\\
 & \frac{a}{a+1} & \frac{1}{a+1}
\end{pmatrix}\,,P_{a}=MN_{a}=\begin{pmatrix}1\\
 & \frac{a}{a+1} & \frac{1}{a+1}\\
 & \frac{a}{a+1} & \frac{1}{a+1}
\end{pmatrix}\,.
\]
We compute $\cD_{K_{a}}(x)$ and $\cD_{\hat{K_{a}}}(N_{a}x)$ for
$x=(3,1,2)^{\mathrm{T}}$. We have $-\langle x\cdot A_{a}x,\pi_{a}\rangle=\frac{24a}{5a+4}$.
Moreover, we have $-\langle P_{a}x\cdot A_{a}P_{a}x,\pi_{a}\rangle=\frac{8a(1+2a)^{2}}{\left(a+1\right)^{2}(5a+4)}\,.$
Hence we have that $\cD_{K_{a}}(x)\geq\cD_{\hat{K}_{a}}(N_{a}x)$
is equivalent to
\begin{align*}
\frac{24a}{5a+4} & \geq\frac{8a(1+2a)^{2}}{\left(a+1\right)^{2}(5a+4)}\ \Leftrightarrow\ (1+2a)^{2}\leq3\left(a+1\right)^{2}\ \Leftrightarrow\ a\leq1+\sqrt{3}=:a_{*}\,.
\end{align*}
In particular, we have for $x=(3,1,2)^{\mathrm{T}}$ that $\cD_{K_{a}}(x)\geq\cD_{\hat{K}_{a}}(N_{a}x)$
for $a\in[0,a^{*}]$ and that $\cD_{K_{a}}(x)\leq\cD_{\hat{K}_{a}}(N_{a}x)$
for $a\in[a_{*},\infty[$. Summarizing, it is not possible to have
uniform estimates between $\cD_{K_{a}}$ and $\cD_{\hat{K}_{a}}(N_{a}\cdot)$.
Hence, no inequality for the Poincaré-type constants can be expected.

\vspace{1.5cm}\noindent \textbf{Acknowledgement:} The research was
supported by Deutsche Forschungsgemeinschaft (DFG) through the Collaborative
Research Center SFB 1114 ‘Scaling Cascades in Complex Systems’ (Project
No. 235221301), Subproject C05 ‘Effective models for materials and
interfaces with multiple scales’.

{\footnotesize{}\bibliographystyle{/Users/astephan/Documents/Science/my_alpha}

\newcommand{\etalchar}[1]{$^{#1}$}
\def\cprime{$'$}

}{\footnotesize\par}
\end{document}